\documentclass[english]{amsart}
\usepackage[T1]{fontenc}
\usepackage[latin9]{inputenc}
\usepackage{amstext}
\usepackage{amsthm}
\usepackage{amssymb}

\makeatletter
\numberwithin{figure}{section}
\theoremstyle{plain}
\newtheorem{thm}{\protect\theoremname}
\theoremstyle{plain}
\newtheorem{fact}[thm]{\protect\factname}
\theoremstyle{plain}
\newtheorem{prop}[thm]{\protect\propositionname}
\theoremstyle{plain}
\newtheorem{lem}[thm]{\protect\lemmaname}

\makeatother

\usepackage{babel}
\providecommand{\factname}{Fact}
\providecommand{\lemmaname}{Lemma}
\providecommand{\propositionname}{Proposition}
\providecommand{\theoremname}{Theorem}

\begin{document}
\title[Obstruction to some Higman embedding theorem]{Obstruction to a Higman embedding theorem for residually finite groups
with solvable word problem}
\begin{abstract}
We prove that, for a finitely generated residually finite group, having
solvable word problem is not a sufficient condition to be a subgroup
of a finitely presented residually finite group. The obstruction is
given by a residually finite group with solvable word problem for
which there is no effective method that allows, given some non-identity
element, to find a morphism onto a finite group in which
this element has a non-trivial image. We also prove that the depth
function of this group grows faster than any recursive function. 
\end{abstract}

\author{Emmanuel Rauzy}
\curraddr{Université de Paris. UFR de Mathématiques. Bâtiment Sophie Germain.
8 place Aurélie Nemours, 75013 Paris, France}
\email{emmanuel.rauzy.14@normalesup.org}
\maketitle

\section*{Introduction }

It is well known that there can be no Higman embedding theorem for
recursively presented finitely generated residually finite groups,
that is to say, that not all finitely generated recursively presented
residually finite groups embed into finitely presented residually
finite groups. Indeed, a theorem of McKinsey (\cite{McKinsey1943,Mostowski1966,Dyson64})
states that all finitely presented residually finite groups have solvable
word problem, while on the other hand several recursively presented
residually finite groups are known, that fail to have solvable word
problem: for instance one example was constructed by Meskin in \cite{Meskin1974},
and one by Dyson in \cite{Dyson1974}. 

It was unknown whether the condition of having solvable word problem
is sufficient for such embeddings to exist. For instance, in the article
\cite{Kharlampovich2017}, the authors ask whether ``unsolvability
of the word problem is the only obstacle'' to embed recursively presented
residually finite groups into finitely presented residually finite
groups. We answer negatively to that question:
\begin{thm}
\label{thm:FirstThm}There exists a finitely generated residually
finite group with solvable word problem, that does not embed in any
finitely presented residually finite group. 
\end{thm}

Call a group $G$ \emph{effectively residually finite }if there is
an algorithm that takes a word $w$ on the generators of $G$ as input,
and, if $w\neq e$ in $G$, produces a morphism $\varphi$ from $G$
to a finite group $F$ that satisfies $\varphi(w)\neq e$. What we
actually prove here is: 
\begin{thm}
\label{thm:Main-THM}There exists a finitely generated residually
finite group with solvable word problem, that is not effectively residually
finite. 
\end{thm}

Theorem \ref{thm:FirstThm} then follows from the following facts: 
\begin{fact}
\label{fact:3}A finitely presented residually finite group is effectively
residually finite. 
\begin{fact}
\label{fact:4}A finitely generated subgroup of an effectively residually
finite group is itself effectively residually finite. 
\end{fact}

\end{fact}

Fact \ref{fact:3} follows from McKinsey's algorithm. It comes from
the fact that a finitely presented group has computable finite quotients,
see \cite{Rauzy2020}. Fact \ref{fact:4} is straightforward. 

In the first section, Dyson's groups are introduced, and they are
used to reduce the proof of Theorem \ref{thm:Main-THM} to a problem
about subsets of $\mathbb{Z}$. In the second section, this problem
is solved. In the last section, we show that the group constructed
in order to prove Theorem \ref{thm:Main-THM} has a depth function
that grows faster than any recursive function. 

Following \cite{Dyson1974}, throughout this article, recursively
presented groups will be called re groups (for recursively enumerable),
and groups in which there is an algorithm that recognizes non-trivial
elements will be called co-re groups. A group has solvable word problem
if and only if it is re and co-re. 

\section{Dyson's Groups }

The author already used Dyson's groups to investigate the property
of having computable finite quotients (\cite{Rauzy2020}). What we
construct here is a strengthening of the result obtained in that first
article, we will include here all definitions but omit the proofs
that already appear there. 

Dyson's groups are amalgamated products of two lamplighter groups. 

The lamplighter group $L$ is the wreath product of $\mathbb{Z}$
and $\mathbb{Z}/2\mathbb{Z}$, noted $\mathbb{Z}\wr\mathbb{Z}/2\mathbb{Z}$,
which is by definition the semi-direct product $\mathbb{Z}\ltimes\underset{\mathbb{Z}}{\bigoplus}\mathbb{Z}/2\mathbb{Z}$,
where $\mathbb{Z}$ acts on $\underset{\mathbb{Z}}{\bigoplus}\mathbb{Z}/2\mathbb{Z}$
by permuting the indices. It admits the following presentation:
\[
\langle a,\varepsilon\vert\,\varepsilon^{2},\,\left[\varepsilon,a^{i}\varepsilon a^{-i}\right],i\in\mathbb{Z}\rangle
\]
The element $a^{i}\varepsilon a^{-i}$ of $L$ corresponds to the
element of $\underset{\mathbb{Z}}{\bigoplus}\mathbb{Z}/2\mathbb{Z}$
with only one non-zero coordinate in position $i\in\mathbb{Z}$. We
call it $u_{i}$. Consider another copy $\hat{L}$ of the lamplighter
group, with an isomorphism we write $g\mapsto\hat{g}$. For each subset
$\mathcal{A}$ of $\mathbb{Z}$, define $L(\mathcal{A})$ to be the
amalgamated product of $L$ and $\hat{L}$, with $u_{i}=a^{i}\varepsilon a^{-i}$
identified with $\hat{u}_{i}=\hat{a}^{i}\hat{\varepsilon}\hat{a}^{-i}$
for each $i$ in $\mathcal{A}$. It has the following presentation:
\[
\langle a,\hat{a},\varepsilon,\hat{\varepsilon}\vert\,\varepsilon^{2},\,\hat{\varepsilon}^{2},\,\left[\varepsilon,a^{i}\varepsilon a^{-i}\right],\left[\hat{\varepsilon},\hat{a}^{i}\hat{\varepsilon}\hat{a}^{-i}\right],i\in\mathbb{Z},\,a^{j}\varepsilon a^{-j}=\hat{a}{}^{j}\hat{\varepsilon}\hat{a}{}^{-j},\,j\in\mathcal{A}\rangle
\]

Whether the group $L(\mathcal{A})$ is residually finite depends on
whether the set $\mathcal{A}$ is closed in the profinite topology
on $\mathbb{Z}$. The profinite topology on a group $G$, which we
denote $\mathcal{PT}(G)$, is the topology defined by taking cosets
of finite index normal subgroups as a basis for open sets. Thus a
subset $\mathcal{A}$ of $\mathbb{Z}$ is open in $\mathcal{PT}(\mathbb{Z})$
if and only if, for every $n$ in $\mathcal{A}$, there exists an
integer $p$ such that $n+p\mathbb{Z}\subseteq\mathcal{A}$. The profinite
topology on $\mathbb{Z}$ was rediscovered in 1955 by Furstenberg
in \cite{Furstenberg1955} (the profinite topology on an arbitrary
group was defined in 1950 in \cite{Hall1950}), where it is used to
give an elegant proof of the existence of infinitely many primes,
and it was proven in \cite{Lovas2010} that this topology comes from
a metric, which is given by the following formula: 

\[
\left\Vert x\right\Vert =\frac{1}{sup\left\{ n:\,1\vert x,2\vert x,3\vert x,...,n\vert x\right\} }
\]
\[
d(x,y)=\left\Vert x-y\right\Vert 
\]
The norm $\left\Vert x\right\Vert $ is thus the reciprocal of the
greatest integer $n$ with the property that $1$, $2$, ..., $n$
all divide $x$. A sequence converges to $0$ in that topology if
and only if, for any integer $k$, there exists a rank from which
on $k$ divides all terms of the sequence. For instance, $n!$ goes
to $0$ as $n$ goes to infinity. Define a function $\theta$ on the
natural numbers by $\theta(n)=\text{lcm}\left\{ 1,2,3...,n\right\} $.
$\theta(n)$ is the smallest non-zero natural number such that $\left\Vert \theta(n)\right\Vert \leq\frac{1}{n}$.
The closed ball of radius $\frac{1}{n}$ and centered in $x$, which
is the set $\left\{ y\in\mathbb{Z},d(x,y)\leq\frac{1}{n}\right\} $,
is simply the set $x+\theta(n)\mathbb{Z}$. It is in fact also open.
Call $\overline{B}(x,r)$ the closed ball centered in $x$ and of
radius $r$, and $B(x,r)$ the corresponding open ball (thought pay
attention that the latter is not the interior of the former). The
distance $d$ is in fact ultrametric: for $x$, $y$ and $z$ integers,
one has $d(x,z)\leq\max(d(x,y),d(y,z))$. This implies that each point
of a ball can be taken as its center, and thus that if two balls intersect,
one is contained in the other. 

Of course, the distance $d$ is effective: $d$ is a recursive function.
This implies that both the closed and open balls of $\mathcal{PT}(\mathbb{Z})$
are recursive sets. 

Call an open set $O$ \emph{effectively open }(in $\mathcal{PT}(\mathbb{Z})$)
if there is an effective procedure that, given any number $n$ in
$O$, computes some integer $p$ such that $n+p\mathbb{Z}$ is contained
in $O$. A set is \emph{effectively closed }if its complement is effectively
open. 

We can now state the properties of the group $L(\mathcal{A})$ that
will allow us to prove Theorem \ref{thm:Main-THM}. 
\begin{prop}
\label{prop:MainProp}Let $\mathcal{A}$ be a subset of $\mathbb{Z}$. 
\begin{enumerate}
\item $L(\mathcal{A})$ is re, co-re or has solvable word problem if and
only if $\mathcal{A}$ is respectively re, co-re or recursive.
\item $L(\mathcal{A})$ is residually finite if and only if $\mathcal{A}$
is closed in $\mathcal{PT}(\mathbb{Z})$. 
\item $L(\mathcal{A})$ is effectively residually finite if and only if
$\mathcal{A}$ is co-re and effectively closed in $\mathcal{PT}(\mathbb{Z})$. 
\end{enumerate}
\end{prop}

(1) and (2) were proved by Dyson in \cite{Dyson1974}. See also \cite{Rauzy2020}.
We prove here only the third point. 
\begin{proof}
Suppose first that $L(\mathcal{A})$ is effectively residually finite.
Then $L(\mathcal{A})$ is co-re and, by (1), $\mathcal{A}$ is co-re
as well. Let $x$ be a integer in the complement of $\mathcal{A}$,
thus such that $u_{x}\hat{u}_{x}^{-1}\neq1$ in $L(\mathcal{A})$.
By our hypothesis, we can effectively find a morphism $\varphi$ from
$L(\mathcal{A})$ to a finite group $F$, with $\varphi(u_{x}\hat{u}_{x}^{-1})\neq1$
in $F$. Call $N$ the product of the orders of the images of $a$
and $\hat{a}$ in $F$. We then claim that $\mathcal{A\cap}(x+N\mathbb{Z})=\varnothing$.
Indeed, if it were not the case, there would exist an integer $k$
such that $x+kN\in\mathcal{A}$, that is to say, such that $u_{x+kN}=\hat{u}_{x+kN}$
in $L(\mathcal{A})$. But then, this would imply: 
\begin{align*}
\varphi(u_{x}) & =\varphi(a)^{kN}\varphi(u_{x})\varphi(a)^{-kN}=\varphi(a^{kN}u_{x}a^{-kN})\\
 & =\varphi(u_{x+kN})=\varphi(\hat{u}_{x+kN})
\end{align*}
Because it can similarly be proved that $\varphi(\hat{u}_{x})=\varphi(\hat{u}_{x+kN})$,
this contradicts the assumption that $\varphi(u_{x}\hat{u}_{x}^{-1})\neq1$.
Thus $\mathcal{A}$ does not meet $x+N\mathbb{Z}$. 

The proof of the converse follows closely the proof of (2) given in
\cite{Rauzy2020} (which differs from the original proof of Dyson),
as one only needs to see that the hypothesis that $\mathcal{A}$ is
effectively closed in $\mathcal{PT}(\mathbb{Z})$ is enough to effectively
carry out that proof. As we actually only use the first implication
of (3) in the proof of Theorem \ref{thm:Main-THM}, no further details
are given here. 
\end{proof}
From Proposition \ref{prop:MainProp}, to prove Theorem \ref{thm:Main-THM},
it suffices to build $\mathcal{A}$ with the following properties:
$\mathcal{A}$ is recursive, $\mathcal{A}$ is closed, but not effectively
so.

\section{Construction in $\mathbb{Z}$}
\begin{lem}
\label{lem:MainLemma}There exists a recursive subset $\mathcal{A}$
of $\mathbb{Z}$, closed in $\mathcal{PT}(\mathbb{Z})$, but not effectively
so. 
\end{lem}

\begin{proof}
We construct a set $\mathcal{B}$, which will be the complement of
the announced $\mathcal{A}$. Thus it has to be recursive and open
but not effectively open.

Call $p_{n}$ the $n$-th prime number. Define a sequence $(t_{n})_{n\geq0}$
by $t_{n}=p_{1}...p_{n}$. This sequence is defined so that $p_{k}$
divides $t_{n}$ if and only if $k\geq n$. Note also that $t_{n}$
divides $t_{n+1}$. 

Consider an effective enumeration of all Turing machines: $M_{1}$
is the first machine, $M_{2}$ is the second... We will build $\mathcal{B}$
as a disjoint union of open sets $X_{n}$, each $X_{n}$ being a neighborhood
of $t_{n}$ defined thanks to a run of the $n$-th Turing machine
$M_{n}$. If this machine does not halt, $X_{n}$ is a closed ball
centered at $t_{n}$ of radius $\frac{1}{t_{n+1}}$. If it halts,
it is a finite union of balls, one of which is centered at $t_{n}$,
the radius of which depends of the number of steps needed for $M_{n}$
to halt. Thanks to this, an information of the form ``$X_{n}$ contains
a ball of radius $r$ centered in $t_{n}$'' will translate in ``if
$M_{n}$ halts, it does so in less than $N$ steps'', where $N$
can be computed from $r$. 

Initialize $X_{n}=\left\{ t_{n}\right\} $. Call $r_{n}=\frac{1}{t_{n+1}}$
and $m=\theta(t_{n+1})$, so that the closed ball $\overline{B}(t_{n},r_{n})$
is the set $t_{n}+m\mathbb{Z}$. 

Start a run of the machine $M_{n}$. 

After each step of computation of $M_{n}$, note $k$ the number of
steps already done in the computation, and add, to $X_{n}$, $t_{n}+km$,
$t_{n}-km$, as well as open balls centered in those numbers, that
are contained in $\overline{B}(t_{n},r_{n})$, and do not contain
$t_{n}$. That is, we replace $X_{n}$ by:
\[
X_{n}\cup B(t_{n}+km,\frac{1}{2}d(t_{n},t_{n}+km))\cup B(t_{n}-km,\frac{1}{2}d(t_{n},t_{n}-km))
\]
Because the distance $d$ is ultrametric, both balls $B(t_{n}+km,\frac{1}{2}d(t_{n},t_{n}+km))$
and $B(t_{n}-km,\frac{1}{2}d(t_{n},t_{n}-km))$ are contained in $\overline{B}(t_{n},r_{n})$.

If, at some point, the machine $M_{n}$ halts, $X_{n}$ consists of
$t_{n}$ and of finitely many open balls centered at points $t_{n}\pm km$.
By construction, a point from $X_{n}\setminus\left\{ t_{n}\right\} $
is at distance at least $\underset{k}{\inf}\left\{ \frac{1}{2}d(t_{n},t_{n}\pm km)\right\} $
from $t_{n}$. This infimum can be computed, call it $r$. Then, compute
the smallest natural number $y$ such that $d(t_{n},y)<r$, and call
$r'$ the distance $d(t_{n},y)$. We then add to $X_{n}$ the ball
$B(u_{n},r')$. This implies that $X_{n}$ cannot contain any ball
of center $t_{n}$ and of radius strictly greater then $r'$, because
it does not contain $y$. In particular, any ball centered in $t_{n}$
that contains one of the elements of the form $t_{n}\pm km$ that
were added to $X_{n}$ is not contained in $X_{n}$. 

Of course, if the machine $M_{n}$ does not halt, $X_{n}$ will be
the whole ball $\overline{B}(t_{n},r_{n})$. 

This ends the definition of $X_{n}$, and $\mathcal{B}$ is defined
as the union $\bigcup X_{n}$. This union is disjoint because, by
the choice of the radius $r_{n}$, any element of $X_{n}$ is divisible
by $p_{n}$, but none is divisible by $p_{n+1}$. We now prove that
$\mathcal{B}$ defined this way satisfies all three properties that
appear in the statement of this Lemma. 

$\mathcal{B}$ is clearly open, because each $X_{n}$ is open, whether
or not the machine $M_{n}$ halts. 

$\mathcal{B}$ is a recursive set. It is obviously recursively enumerable,
because it was defined by an effective enumeration. To see that it
is also co-re, let $x$ be an integer, we want to decide whether $x$
belongs to $\mathcal{B}$. By looking at the prime decomposition of
$x$, one can find up to one $n$ such that $x$ might belong to $X_{n}$.
Because $X_{n}$ is always contained in $\overline{B}(t_{n},r_{n})$,
if $d(t_{n},x)>r_{n}$, $x$ cannot be in $X_{n}$. Otherwise, it
belongs to $\overline{B}(t_{n},r_{n})=t_{n}+m\mathbb{Z}$, and we
can find $k$ such that $x=t_{n}+km$. Then, if a run of $M_{n}$
lasts more then $k$ steps, automatically $x$ will belong to $X_{n}$.
On the other hand, if $M_{n}$ stops in less than $k$ steps, $X_{n}$
can be determined explicitly as a finite union of open balls, and
thus the question of whether $x$ belongs to $X_{n}$ can be settled.
Because the problem ``does $M_{n}$ halt in more then $k$ steps''
is a computable one, in either case we will be able to determine whether
or not $x$ belongs to $\mathcal{B}$. 

Finally, $\mathcal{B}$ is not effectively open. Suppose we have an
algorithm that gives, for $x$ in $\mathcal{B}$, an integer $k$
such that $x+k\mathbb{Z}$ is contained in $\mathcal{B}$. Applying
it to $t_{n}$, we can find a radius $r$ such that $B(t_{n},r)$
is contained in $\mathcal{B}$. We will show that this information
implies a new information of the form: if the machine $M_{n}$ halts,
then it halts in less than $N$ steps. This would of course allow
one to solve the halting problem, thus such an algorithm does not
exist. 

Indeed, we have seen that if $M_{n}$ halts in $N$ steps, $X_{n}$
cannot contain any ball centered in $t_{n}$ that contains an element
of the form $t_{n}\pm km$, with $k\leq N$. Turning this around,
computing $N$ such that $t_{n}+mN$ belongs to $B(t_{n},r)$, (for
instance $N=\theta(\left\lceil \frac{1}{r}\right\rceil +1)$), the
information ``$B(t_{n},r)$ is contained in $\mathcal{B}$'' implies
that either $M_{n}$ does not halt, or it halts in less then $N$
steps. 

This ends the proof of Lemma \ref{lem:MainLemma}.
\end{proof}
What we do not know yet is whether the condition of being recursively
presented and effectively residually finite, which is necessary to
be a subgroup of a finitely presented residually finite group, is
also sufficient. For instance, the strictly stronger condition of
having computable finite quotients is not known to be unnecessary
(see \cite{Rauzy2020}). 

\section{Non-recursive depth function}

In \cite{Bou-Rabee2010}, Bou-Rabee introduced the\emph{ residual
finiteness growth function}, or \emph{depth function}, $\rho_{G}$,
of a residually finite group $G$. To a natural number $n$, $\rho_{G}$
associates the smallest number $k$ such that, for any non-trivial
element of length at most $n$ in $G$, there exists a finite quotient
of $G$ of order at most $k$, such that the image of this element
in that quotient is non-trivial.

A finitely generated subgroup $H$ of a finitely generated residually
finite group $G$ must have a depth function $\rho_{H}$ that grows
slower than that of $G$ (see \cite{Bou-Rabee2010}). Because it is
easy to see that a finitely presented residually finite group always
has a recursive growth function, a subgroup of a finitely presented
residually finite group must have its depth function bounded above
by a recursive function. Note that in \cite{Kharlampovich2017}, it
was shown that finitely presented residually finite groups can have
arbitrarily large recursive depth function. It was also known that
the depth function of a residually finite group could grow arbitrarily
fast (see \cite{Bou-Rabee2016}).

Note that an effectively residually finite group with solvable word
problem always has its depth function bounded above by a recursive
function. 

We now show:
\begin{prop}
Let $\mathcal{A}$ be the subset of $\mathbb{Z}$ given in Lemma \ref{lem:MainLemma}.
The depth function $\rho_{L(\mathcal{A})}$ of $L(\mathcal{A})$ cannot
be smaller than a recursive function. 
\end{prop}

\begin{proof}
Suppose there exists a recursive function $f$ such that $\rho_{L(\mathcal{A})}\leq f$.
Then, for each $n$ which is not in $\mathcal{A}$, $u_{n}\hat{u}_{n}^{-1}$
is a non-identity element of $L(\mathcal{A})$, and thus it has a
non-trivial image in a finite quotient $F$ of size at most $f(4n+2)$
(because $u_{n}$ and $\hat{u}_{n}$ are of word length $2n+1$).
Now the orders of $a$ and $\hat{a}$ in $F$ both divide $f(4n+2)!$,
thus $F$ is a quotient of the quotient of $L(\mathcal{A})$ given
by adding to it the two relations $a^{f(4n+2)!}=e$ and $\hat{a}^{f(4n+2)!}=e$
(i.e. $\langle L(\mathcal{A})\,\vert\,a^{f(4n+2)!},\,\hat{a}^{f(4n+2)!}\rangle$).
Thus, in this group as well, $u_{n}\hat{u}_{n}^{-1}$ is a non-identity
element, and we know that this implies that $n+f(4n+2)!\mathbb{Z}$
does not meet $\mathcal{A}$ (see the proof of Proposition \ref{prop:MainProp}).
This shows that $\mathcal{A}$ is effectively closed, contradicting
our hypothesis. 
\end{proof}
\bibliographystyle{abbrv}
\bibliography{HigmanObstruction.bib}

\end{document}